\documentclass[11pt]{amsart}

\usepackage{amsmath}
\usepackage[font=small]{caption}
\usepackage{amsthm}
\usepackage{amsfonts}
\usepackage[pdftex]{graphicx}
\usepackage[margin=.5cm, format=hang]{subfig}
\usepackage[alphabetic]{amsrefs}
\usepackage{amssymb}
\usepackage{tikz-cd}
\usepackage{pinlabel}
\usepackage{morefloats}
\usepackage{enumitem}
\usepackage{fullpage}
\usepackage{array}
\usepackage{multirow}
\usepackage{sidecap}
\usepackage{hhline}
\usepackage{multirow}
\usepackage{blkarray}
\usepackage{xcolor}
\usepackage{comment}

\theoremstyle{plain}
\newtheorem{theorem}{Theorem}[section]
\newtheorem{proposition}[theorem]{Proposition}
\newtheorem{lemma}[theorem]{Lemma}

\newtheorem{conjecture}[theorem]{Conjecture}

\theoremstyle{definition}

\theoremstyle{remark}
\newtheorem{remark}[theorem]{Remark}

\newtheorem{question}[theorem]{Question}

\newcommand{\Z}{\mathbb{Z}}

\newcommand{\wh}{\widehat}

\numberwithin{equation}{section}

\author[Tye Lidman]{Tye Lidman}
\address{Department of Mathematics, North Carolina State University}
\email{tlid@math.ncsu.edu}

\author[Trevor Oliveira-Smith]{Trevor Oliveira-Smith}

\address{Department of Mathematics, University of California, Davis}
\email{tdoliveirasmith@ucdavis.edu}

\author[Alexander Zupan]{Alexander Zupan}
\address{Department of Mathematics, University of Nebraska - Lincoln}
\email{zupan@unl.edu}

\title{A homological generalized Property R conjecture is false}
\date{}

\begin{document}
\maketitle

\begin{abstract}
The generalized Property R conjecture (GPRC) predicts that if framed surgery on an $n$-component link $L$ in $S^3$ produces $\#^{n} (S^1\times S^2)$, then $L$ is handleslide equivalent to an unlink, the obvious way to construct such a surgery.  Many potential counterexamples to the GPRC are known, but obstructing handleslide equivalence is a tricky proposition.  In this vein, we disprove a further generalization of the GPRC.  It would be reasonable to expect that if an $n$-component link in $S^3$ surgers to the connected sum of $n$ three-manifolds with the homology of $S^1 \times S^2$, then this link should be handleslide equivalent to an $n$-component split link, the obvious way to construct such a surgery.  However, we prove that there are 2-component framed links in $S^3$ that surger to a connected sum of homology $S^1\times S^2$'s but that are not handleslide equivalent, or even weakly handleslide equivalent, to a split link. 
\end{abstract}

\section{Introduction}

Using the theory of foliations, Gabai proved the famous Property R conjecture: 
\begin{theorem}[\cite{Ga}]
Let $K$ be a knot in $S^3$.  If Dehn surgery on $K$ yields $S^1 \times S^2$, then $K$ is unknotted.
\end{theorem} 
In other words, the only way to get $S^1 \times S^2$ via surgery on a knot in $S^3$ is the obvious one.  An attempt to extend Gabai's result to links encounters a subtlety:  Although 0-surgery on an $n$-component unlink produces $\#^{n} (S^1\times S^2)$, this is not the unique $n$-component surgery description of $\#^{n} S^1\times S^2$, as we can perform handleslides on the unlink to produce more complicated $n$-component links with the same surgery.  One then optimistically proposes the following:
\begin{conjecture}[Problem $1.82$ \cite{kirby-web}, generalized Property R conjecture (GPRC)]
Let $L$ be an $n$-component framed link in $S^3$.  If surgery on $L$ produces $\#^{n} (S^1\times S^2)$, then $L$ is handleslide equivalent to the unlink.
\end{conjecture}

Potential counterexamples to the GPRC were given in \cite{GST} and extended to broader families of possible counterexamples in \cite{MeZ}.  These links have stubbornly resisted progress in either direction; they do not appear to admit the desired handleslides, while obstructing handleslide equivalence is delicate and seemingly intractable.  Naturally, when one encounters a difficult conjecture, there are two paths forward:  Add hypotheses until the conjecture becomes true, or weaken them until the conjecture becomes false.  In this note, we relax restrictions on the result of surgery to explore whether a broadening of the GPRC is possible.  If $L$ is an $n$-component split link, performing 0-surgery on each component of $L$ yields a connected sum of $n$ three-manifolds, each a homology $S^1 \times S^2$.  In the spirit of the GPRC, it is natural to conjecture (up to handleslides) that this construction is the only way to concoct such a surgery.  However, we prove

\begin{theorem}\label{thm:main1}
There exists an infinite family of 2-component links $\{L_n\}$ with the following property:  The 0-surgery on $L_n$ produces a three-manifold of the form $Y_n \# Z_n$ with $H_1(Y_n) \cong H_1(Z_n) \cong \mathbb{Z}$, while $L_n$ is not handleslide equivalent (as a 0-framed link) to a split link.
\end{theorem}

The theorems above have a deeper 4-dimensional perspective as well.  Gabai's result implies that a smooth homotopy ball built out of a single 1- and 2-handle must necessarily be diffeomorphic to the standard four-ball.  The GPRC would imply that a smooth homotopy 4-sphere with a handle decomposition with no 1-handles is diffeomorphic to the standard four-sphere.  Such a homotopy 4-sphere is called \emph{geometrically simply-connected}.  While it is still unknown whether the links in \cite{GST} and \cite{MeZ} are counterexamples to the GPRC, it was shown that these links give rise to standard $S^4$'s (see also \cite{gompf-killing}).  The authors of \cite{GST} state a weaker version of the GPRC, which has the same $4$-dimensional conclusion but allows for the introduction of cancelling handle-pairs.

Two framed links are said to be \emph{weakly handleslide equivalent} if they are related by handleslides, the addition or deletion of split 0-framed unknots, and the addition or deletion of split Hopf pairs (a Hopf link consisting of one 0-framed unknot and one dotted unknot, representing a 4-dimensional 1-handle).  Further details can be found in~\cite{GS} or~\cite{MSZ}.

\begin{conjecture}[\cite{GST}, weak GPRC]
Let $L$ be an $n$-component framed link in $S^3$.  If surgery on $L$ produces $\#^{n} S^1\times S^2$, then $L$ is weakly handleslide equivalent to an unlink.
\end{conjecture}

By Cerf~\cite{cerf}, the weak GPRC is equivalent to the assertion that every geometrically simply-connected homotopy 4-sphere is diffeomorphic to the standard 4-sphere.  In this context, we prove

\begin{theorem}\label{thm:main2}
There exists an infinite family of 2-component links $\{L_n\}$ with the following property:  The 0-surgery on $L_n$ produces a three-manifold of the form $Y_n \# Z_n$ with $H_1(Y_n) \cong H_1(Z_n) \cong \mathbb{Z}$, while $L_n$ is not weakly handleslide equivalent (as a 0-framed link) to a 2-component split link.
\end{theorem}

Theorem~\ref{thm:main1} follows immediately from Theorem~\ref{thm:main2}.  To prove the theorem, we note that weak handleslide equivalence preserves the result of surgery on a link, up to adding or deleting $S^1 \times S^2$ summands.  Thus, if a 2-component link $L$ with surgery to $Y \# Z$, where $Y$ and $Z$ are homology $S^1 \times S^2$'s, is weakly handleslide equivalent to a split link $K_1 \sqcup K_2$, then each summand of the manifold $Y \# Z$ obtained by surgery on $L$ is obtained by surgery on $K_1$ or $K_2$.  Hedden, Kim, Mark, and Park proved that for $n$ odd, a certain family $\{Y_n\}$ of Seifert fibered spaces cannot be obtained by surgery on a knot in $S^3$~\cite{HKMP}, and Johnson extended their work to the even case~\cite{johnson}.  Thus, we complete the proof of the theorem by constructing links $L_n$ with surgeries to $Y_n \# Z_n$, where $Y_n$ lives in this impermissible family.

Given a framed knot $K$ in $S^3$, the \emph{knot trace} $X_K$ is the compact 4-manifold constructed by adding attaching a 2-handle to $B^4$ along $K$ via its framing.  For a framed link $L$, the \emph{link trace} $X_L$ is defined similarly, attaching one 2-handle for each component of $L$.  With this perspective, two links have diffeomorphic traces if and only if they are weakly handleslide equivalent~\cite{cerf}.  For the links $L_n$ constructed in Theorem~\ref{thm:main2}, our proof implies that $X_{L_n}$ is not a boundary connected sum of two knot traces.  More generally, $X_{L_n}$ cannot be expressed smoothly as $X_1 \natural X_2$ with $\partial X_1 = Y_n$ and $\partial X_2 = Z_n$.  This is because the $X_i$ would necessarily be homotopy $S^2 \times D^2$'s, which would contradict the work of \cite{HKMP, johnson}.  However, we can still ask the topological analogue:

\begin{question}
Is $X_{L_n}$ homeomorphic to the boundary sum of two topological four-manifolds $X_1$ and $X_2$ with $\partial X_1 = Y_n$ and $\partial X_2 = Z_n$?
\end{question}

%Given a framed knot $K$ in $S^3$, the \emph{knot trace} $X_K$ is the compact 4-manifold constructed by adding attaching a 2-handle to $B^4$ along $K$ via its framing.  For a framed link $L$, the \emph{link trace} $X_L$ is defined similarly, attaching one 2-handle for each component of $L$.  With this perspective, two links have diffeomorphic traces if and only if they are weakly handleslide equivalent~\cite{cerf}.  For the links $L$ constructed in Theorem~\ref{thm:main2}, our proof implies that $X_L$ is not a boundary connected sum of two knot traces.  However, we can still ask the following:

%\begin{question}
%For the links $L_n$ in Theorem~\ref{thm:main2}, where surgery yields $Y_n \# Z_n$, the corresponding link traces $X_{L_n}$ cannot be expressed smoothly as $X_1 \natural X_2$ with $\partial X_1 = Y_n$ and $\partial X_2 = Z_n$.  This is because the $X_i$ would necessarily be homotopy $S^2 \times D^2$'s and this would contradict the work of \cite{HKMP, johnson}.  But, is $X_{L_n}$ homeomorphic to the boundary sum of two topological four-manifolds $X_1$ and $X_2$ with $\partial X_1 = Y_n$ and $\partial X_2 = Z_n$?
%For the links $L_n$ in Theorem~\ref{thm:main2}, where surgery yields $Y_n \# Z_n$, can the corresponding link traces $X_{L_n}$ be expressed as $X_1 \natural X_2$ with $\partial X_1 = Y_n$ and $\partial X_2 = Z_n$?  In other words, does the connect sum 3-sphere in $\partial X_{L_n}$ bound a 4-ball in $X_{L_n}$? 
%\end{question}

Finally, turning back to the connections with the GPRC, we wonder if a refinement of our results is possible, especially in light of the fact that the candidate links from~\cite{GST} and~\cite{MeZ} satisfy the weak GPRC but are not known to satisfy the GPRC.

\begin{question}
Is there a 2-component framed link $L$ such that surgery on $L$ produces a three-manifold of the form $Y \# Z$ with $H_1(Y) \cong H_1(Z) \cong \mathbb{Z}$ and $L$ is weakly handleslide equivalent, but not handleslide equivalent, to a 2-component split link?
\end{question}

\section*{Outline} In Section~\ref{sec:prelim} we discuss some basics of Seifert fibered spaces and assemble the ingredients we need to prove the main theorem.  In Section~\ref{sec:proof}, we construct the relevant links and prove Theorem~\ref{thm:main2}.  In Section~\ref{sec:proof2}, we provide an alternative perspective on the links $\{L_n\}$.  

\section*{Acknowledgements} TL and AZ thank the organizers of the Midwest Panorama of Geometry and Topology at the University of Iowa in June 2025, where part of this work was completed.  AZ also thanks Jeffrey Meier for helpful conversations related to this work.  TL was supported in part by NSF grants DMS-2105469 and DMS-2506277 and a Simons Travel Support award.  TOS was supported as a GAANN Fellow through the Mathematics at UC Davis GAANN grant funded by the Department of Education grant number P200A240025.  AZ was supported in part by NSF grant DMS-2405301 and a Simons Travel Support award.

\section{Preliminaries}\label{sec:prelim}

We work in the smooth category, and we use $\eta( \cdot )$ to denote a regular neighborhood.  Given a link $L$ in $S^3$, a \emph{framing} for $L$ is a choice of an extended rational number $a/b$ for each component $K$ of $L$, expressed as $a[\mu] + b[\ell] \in H_1(\partial \eta(K))$, where $\eta(K)$ is a regular neighborhood of $K$, $\mu$ is a meridian of the solid torus $\eta(K)$, and $\ell$ is the Seifert longitude.  For a framed link $L$, the result of \emph{Dehn surgery} on $L$ is the 3-manifold obtained by removing neighborhoods of each component and reattaching solid tori via maps sending meridians to the curves $a[\mu] + b[\ell]$ corresponding to the framings of each component.  A \emph{surgery diagram} for a 3-manifold $Y$ consists of a framed link $L$ such performing the associated surgery on each component results in $Y$.  For further details, see~\cite{GS}, for instance.

\begin{remark}
Our use of ``framing" is a slight abuse of notation/language.  In the literature, ``framing" typically refers to an integral framing of the form $a/1$, and our use of ``framing" might instead be called a ``surgery coefficient" or ``slope."
\end{remark}

We have the following elementary lemma.

\begin{lemma}\label{lem:alg-unlink} 
Let $L = K\cup J$ be a framed two-component link in $S^3$ with a surgery to a twice connect sum of homology $S^1\times S^2$'s. Then the components of $L$ are algebraically unlinked and $0$-framed.
\end{lemma}
\begin{proof}
Let $Y \# Z$ denote the result of our prescribed surgery on $L$. The Meyer-Vietoris sequence immediately gives us $H_{1}(S^3\setminus\eta(L))\cong \Z^2$, with generators given by the meridians of the components of $L$. By filling in the solid tori via the specified framing, the homomorphism induced by the inclusion $$i_*: H_1(S^3\setminus\eta(L))\to H_{1}(Y \# Z)$$ is an epimoprhism. Since $H_1(Y \# Z) = \Z^2$, we may conclude that $i_*$ is actually an isomorphism. For each toroidal component $T$ of $\partial\eta(L)$, the solid torus we fill in will kill some generator of $H_{1}(T)$. Thus, the induced homomorphism $H_{1}(T)\to H_{1}(Y \# Z)$ has non-trivial kernel containing the framing curve. But the induced inclusion $H_{1}(T)\to H_{1}(S^3\setminus\eta(L))$ has kernel generated by the Seifert longitude. Looking at the induced composition homomorphism, it must follow that the framing on each component is the Seifert longitude, i.e the $0$-framing. Since each Seifert longitude is null-homologous in the complement $S^3\setminus\eta(L)$, the linking numbers must all be trivial.
\end{proof}

Next, we give a brief treatment of Seifert fibered spaces.  We follow the notation of Chapter $10$ of \cite{Mart}.  The Seifert fibered space $(S^2; \frac{p_1}{q_1}, \ldots, \frac{p_k}{q_k})$, which is denoted $(S^2, (p_1,q_1),\dots,(p_k,q_k))$ in~\cite{Mart}, has base orbifold $S^2$ and $k$ singular fibers of orders $p_1,\ldots, p_k$, and it is represented by the surgery diagram in Figure~\ref{fig:seifert}.

\begin{figure}[h!]
    \centering
    \includegraphics[width=0.5\textwidth]{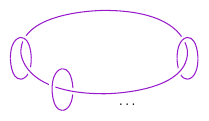}
        \put (-236,55) {$\frac{p_1}{q_1}$}
        \put (-190,20) {$\frac{p_2}{q_2}$}
        \put (-7,55) {$\frac{p_k}{q_k}$}
        \put (-120,128) {$0$}
\caption{Surgery description for the Seifert fibered space $(S^2; \frac{p_1}{q_1},\ldots, \frac{p_k}{q_k})$}
\label{fig:seifert}
\end{figure}
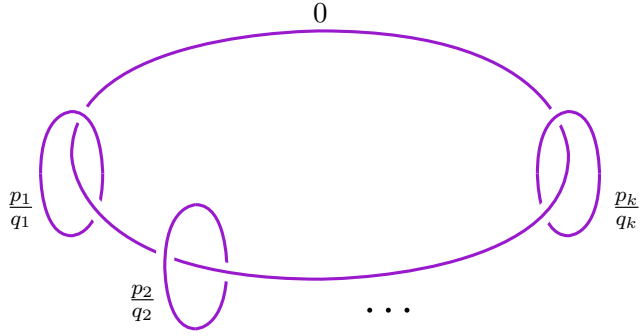

From~\cite{Mart}, we have a useful lemma.

\begin{lemma}\label{lem:S3}
The Seifert fibered space $\left(S^2; \frac{p_1}{q_1},\frac{p_2}{q_2}\right)$ is $S^3$ if and only if $p_1q_2 + p_2q_1 = \pm 1$.
\end{lemma}

\begin{proof}
The space $Y = \left(S^2; \frac{p_1}{q_1},\frac{p_2}{q_2}\right)$ is given by surgery on two regular fibers in $S^1 \times S^2$, and thus $Y$ is the union of two Seifert fibered solid tori.  On a fixed Heegaard torus separating the two solid tori, one meridian is a $(p_1,q_1)$-curve, while the other is a $-(p_2,q_2)$-curve.  The intersection number of these two curves is given by $p_1q_2 + q_1p_2$.  It follows that $Y = S^3$ exactly when this intersection number is $\pm 1$.
\end{proof}

The \emph{Euler number} of a Seifert fibered space $Y = (S^2;\frac{p_1}{q_1},\dots,\frac{p_k}{q_k})$ is given by $e(Y) = \frac{q_1}{p_1} + \dots +\frac{q_k}{p_k}$.  Notably, if $Y$ is a homology $S^1 \times S^2$, then $e(Y) = 0$~\cite[Proposition 10.3.5]{Mart}.

It will be helpful to describe a knot $K \subset S^3$ by drawing an additional unlabeled curve in a surgery diagram for $S^3$.  For example, by Lemma~\ref{lem:S3}, we have that any $p$ and for any odd $q$, with $r = \frac{q+1}{2}$, both $(S^2; \frac{p}{1},-\frac{p-1}{1})$ and $(S^2; \frac{2}{1},-\frac{r}{q})$ represent $S^3$, but with different Seifert fibrations.  Thus the curves shown in Figure~\ref{fig:torus} represent knots in $S^3$.  In particular, in both cases, the knot is contained in a Heegaard torus for $S^3$, and with these parameters, represent a $-(p,p-1)$-torus knot $-T_{p,p-1}$ or a $-(q,2)$-torus knot $-T_{q,2}$, respectively.

\begin{figure}[h!]
    \centering
    \includegraphics[width=0.3\textwidth]{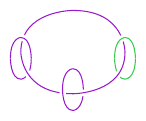} \qquad \qquad
        \includegraphics[width=0.3\textwidth]{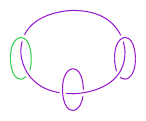}
        \put (-330,50) {$p$}
        \put (-285,-3) {$-(p-1)$}
        \put (-268,112) {$0$}
        \put (-5,50) {$2$}
        \put (-60,2) {$\frac{r}{q}$}
        \put (-70,112) {$0$}

\caption{A $-(p,p-1)$-torus knot (left) and $-(q,2)$-torus knot (right) in $S^3$}
\label{fig:torus}
\end{figure}
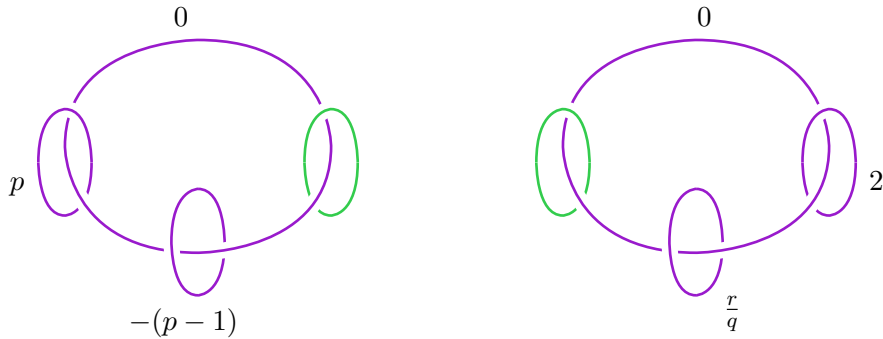

A critical ingredient in our proof of Theorem~\ref{thm:main2} comes from~\cite{HKMP} and~\cite{johnson}, in which they obstruct certain 3-manifolds from being obtained by surgery on a knot in $S^3$.  The case of $n$ odd is completed in~\cite{HKMP} and extended to $n$ even in~\cite{johnson}.

\begin{theorem}\cite{HKMP,johnson}\label{thm:deus}
For any natural number $n$, the homology $S^1 \times S^2$ given by
\[Y_n = \left(S^2; -\frac{2}{1}, \frac{8n-1}{1}, \frac{16n-2}{8n-3}\right)\]
 is not obtained by 0-surgery on a knot in $S^3$.
\end{theorem}

\begin{remark}
The manifolds obstructed in Theorem 1.2.1 of~\cite{johnson} are identified as
\[A_n = \left(S^2;-\frac{2}{1}, \frac{-8n+1}{1},\frac{16n-2}{8n+1}\right),\]
whose parameters are not identical to those given in Theorem~\ref{thm:deus}.  Nevertheless, we can see that $A_n = (S^2;\frac{2}{1}, \frac{-8n+1}{1},-\frac{16n-2}{8n-3}) = -Y_n$ (see e.g. Proposition 10.3.11 from~\cite{Mart} on the equivalence of Seifert fibered spaces) and thus $A_n$ is obtained by surgery on a knot $K$ in $S^3$ if and only if $Y_n$ is obtained by surgery on $-K$ in $S^3$.
\end{remark}

Finally, to construct the links from Theorem~\ref{thm:main2}, we use manipulations of a surgery diagram without altering the corresponding 3-manifold $Y$.  The collection of these moves is often referred to as \emph{Kirby calculus}, and three of them are described below and shown in Figure~\ref{fig:kirby}.

\begin{enumerate}
\item We can cancel any component with integral surgery coefficient and a 0-framed meridian.
\item If a 0-framed unknot bounds a disk meeting exactly two components with integer framings $n$ and $m$ each in a single point, we can band these two components along the disk and cancel the unknotted component.  The new component has framing $n+m$.
\item If one $r$-framed component $K_1$ (with $r \in \mathbb{Q}$ and $r \neq 0$) is a meridian of another $n$-framed component $K_2$, where $n \in \Z$, we can cancel $K_1$ and endow $K_2$ with framing $n - \frac{1}{r}$.
\end{enumerate}

\begin{figure}[h!]
    \centering
    \phantom{hi} \includegraphics[width=0.3\textwidth]{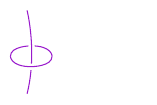}
        \includegraphics[width=0.3\textwidth]{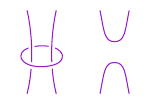} \qquad
                \includegraphics[width=0.3\textwidth]{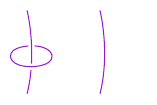}
        \put (-400,36) {$0$}
        \put (-384,42) {$\sim \quad \emptyset$}
        \put (-250,30) {$0$}
        \put (-291,70) {$n$}
        \put (-256,70) {$m$}
        \put (-189,70) {$n+m$}
        \put (-237,42) {$\sim$}
        \put (-95,30) {$r$}
        \put (-110,70) {$n$}
        \put (-43,70) {$n-\frac{1}{r}$}
        \put (-73,42) {$\sim$}
\caption{Moves of type (1) at left, type (2) at center, and type (3) at right.}
\label{fig:kirby}
\end{figure}
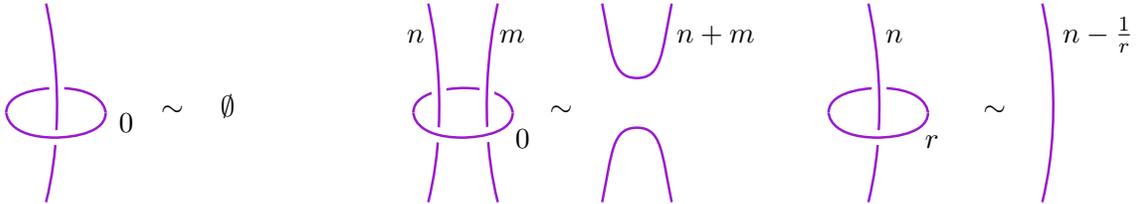

These moves are well-known; the move of type (1) can be found for example in Proposition 3.3 of~\cite{Sav}, the move of type (2) is the example shown in Figure 3.24 of~\cite{Sav}, and the move of type (3) is a ``slam-dunk", shown in Figure 5.30 of~\cite{GS}.  More details about Kirby calculus can be found in~\cite{GS} or~\cite{Sav}, for example.

\section{Proof of the Main Theorem}\label{sec:proof}

In this section, we establish Proposition~\ref{prop:surg-desc} below, constructing links $L_n$, the 0-surgeries on which yield $Y_n \# Z_n$, where $Y_n$ corresponds to the Seifert fibered spaces from Theorem~\ref{thm:deus} above.  Theorem~\ref{thm:main2} quickly follows.

\begin{proposition}\label{prop:surg-desc}
For each $n \geq 1$, there exists a 2-component link $L_n$ with the property that 0-surgery on $L_n$ is $Y_n \# Z_n$, where $H_1(Y_n) = H_1(Z_n) = \Z$, $Y_n = (S^2; -\frac{2}{1}, \frac{8n-1}{1}, \frac{16n-2}{8n-3})$, and $Z_n = (S^2;  -\frac{16n-2}{8n-3}, -\frac{8n-2}{1}, \frac{32n^2-12n+1}{16n^2-6n+1})$.
\end{proposition}

\begin{proof}
First, consider the unlabeled (green) knot $K_n$ in Figure~\ref{fig:sumknot}.  Ignoring $K_n$, the corresponding 3-manifold is $S^3 \# S^3 = S^3$ by Lemma~\ref{lem:S3}.  In addition, $K_n$ can be obtained by taking the connected sum of the two knots in Figure~\ref{fig:torus} (with one mirrored), so that
\[ K_n = -T_{8n-1,8n-2} \# T_{32n^2-12n+1,2}.\]

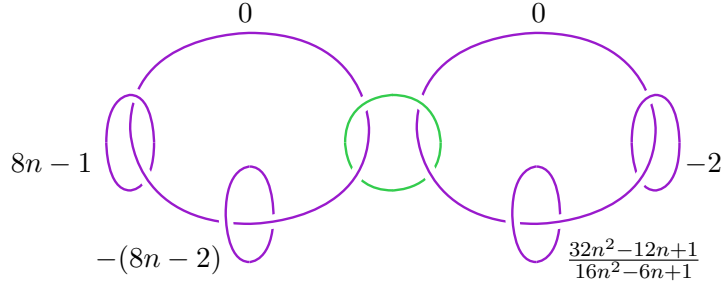
\begin{figure}[h!]
    \centering
    \phantom{hi} \includegraphics[width=0.5\textwidth]{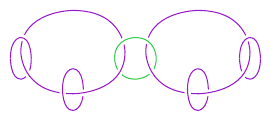}
        \put (-262,43) {$8n-1$}
        \put (-230,7) {$-(8n-2)$}
        \put (-176,100) {$0$}
        \put (-7,43) {$-2$}
        \put (-52,7) {$\frac{32n^2-12n+1}{16n^2-6n+1}$}
        \put (-65,100) {$0$}
\caption{The knot $K_n$ in $S^3$, unlabeled and shown in green.}
\label{fig:sumknot}
\end{figure}

Let $M_n$ denote the result of 0-surgery on $K_n$.  Then performing move (2), we see the surgery diagram for $M_n$ shown at left in Figure~\ref{fig:fourfibermassage}, from which it follows that $M_n = (S^2; -\frac{2}{1}, \frac{8n-1}{1}, -\frac{8n-2}{1}, \frac{32n^2-12n+1}{16n^2-6n+1})$.  The framed link at left is isotopic to the framed link at right in Figure~\ref{fig:fourfibermassage}, which is equivalent to the framed link in Figure~\ref{fig:chain} via a sequence of four consecutive moves of type (2).

\begin{figure}[h!]
    \centering
\includegraphics[width=0.35\textwidth]{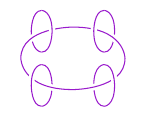} \qquad
    \includegraphics[width=0.35\textwidth]{figures/fourfiber1.pdf}
        \put (-358,100) {$8n-1$}
        \put (-368,15) {$-(8n-2)$}
        \put (-275,106) {$0$}
        \put (-222,100) {$-2$}
        \put (-224,15) {$\frac{32n^2-12n+1}{16n^2-6n+1}$}
        \put (-164,100) {$8n-1$}
        \put (-146,15) {$-2$}
        \put (-32,100) {$-(8n-2)$}
        \put (-32,15) {$\frac{32n^2-12n+1}{16n^2-6n+1}$}
        \put (-83,106) {$0$}
\caption{Isotopic surgery diagrams for $M_n$.}
\label{fig:fourfibermassage}
\end{figure}

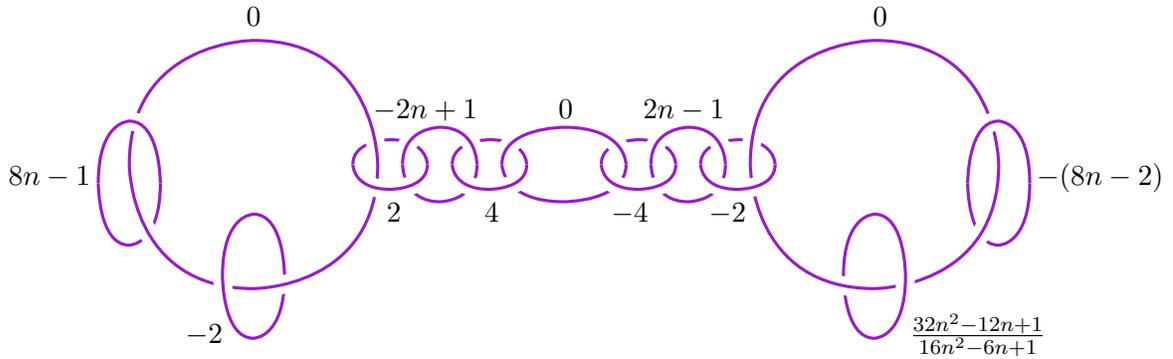
\begin{figure}[h!]
    \centering
    \includegraphics[width=0.8\textwidth]{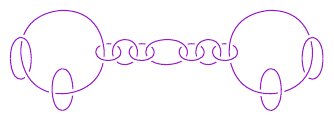}
        \put (-398,70) {$8n-1$}
        \put (-331,10) {$-2$}
        \put (-308,130) {$0$}
        \put (-255,56) {$2$}
        \put (-260,95) {$-2n+1$}
        \put (-218,56) {$4$}
        \put (-190,95) {$0$}
        \put (-170,56) {$-4$}
        \put (-158,95) {$2n-1$}
        \put (-133,56) {$-2$}
        \put (-9,70) {$-(8n-2)$}
        \put (-57,10) {$\frac{32n^2-12n+1}{16n^2-6n+1}$}
        \put (-71,130) {$0$}
\caption{Another surgery diagram for $M_n$.  Four moves of type (2) converts this link to the one in Figure~\ref{fig:fourfibermassage}.}
\label{fig:chain}
\end{figure}

Let $J_n$ denote the knot in $M_n$ shown in Figure~\ref{fig:jn}.  We may isotope $J_n$ in $M_n$ to miss the core curve from the surgery on $K_n$, in which case we can view $L_n = K_n \cup J_n$ as a 2-component link in $S^3$.  Let $N_n$ denote the result of 0-surgery on $J_n$ in $M_n$ (that is, $N_n$ is the result of 0-surgery on $L_n$ in $S^3$).  A surgery diagram for $N_n$ is obtained by giving $J_n$ the 0-framing, and a move of type (1) converts this surgery diagram to the one at the top of Figure~\ref{fig:surgeryJn}, which four moves of type (3) convert to the surgery diagram at bottom, where the new framings are verified by the computation below (and its opposite):
\[ 2 - \frac{1}{(-2n+1)-\frac{1}{4}} = 2 + \frac{4}{8n-3} = \frac{16n-2}{8n-3}.\]
We conclude that
\[ N_n = \left(S^2; -\frac{2}{1}, \frac{8n-1}{1}, \frac{16n-2}{8n-3}\right) \# \left(S^2;  -\frac{16n-2}{8n-3}, -\frac{8n-2}{1}, \frac{32n^2-12n+1}{16n^2-6n+1}\right) = Y_n \# Z_n,\]
as desired.
\end{proof}

\begin{figure}[h!]
    \centering
    \includegraphics[width=0.8\textwidth]{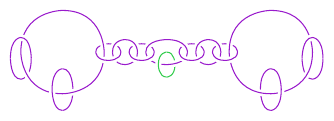}
        \put (-398,70) {$8n-1$}
        \put (-331,10) {$-2$}
        \put (-308,130) {$0$}
        \put (-255,56) {$2$}
        \put (-260,95) {$-2n+1$}
        \put (-218,56) {$4$}
        \put (-190,95) {$0$}
        \put (-170,56) {$-4$}
        \put (-158,95) {$2n-1$}
        \put (-133,56) {$-2$}
        \put (-9,70) {$-(8n-2)$}
        \put (-57,10) {$\frac{32n^2-12n+1}{16n^2-6n+1}$}
        \put (-71,130) {$0$}
\caption{The knot $J_n$ in $M_n$.}
\label{fig:jn}
\end{figure}

\begin{figure}[h!]
    \centering
    \includegraphics[width=0.8\textwidth]{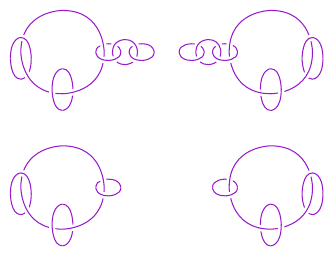}
        \put (-400,215) {$8n-1$}
        \put (-334,170) {$-2$}
        \put (-307,283) {$0$}
        \put (-255,210) {$2$}
        \put (-256,248) {$-2n+1$}
        \put (-216,210) {$4$}
        \put (-173,210) {$-4$}
        \put (-157,248) {$2n-1$}
        \put (-133,210) {$-2$}
        \put (-10,215) {$-(8n-2)$}
        \put (-55,170) {$\frac{32n^2-12n+1}{16n^2-6n+1}$}
        \put (-70,283) {$0$}
        \put (-400,62) {$8n-1$}
        \put (-334,17) {$-2$}
        \put (-307,130) {$0$}
        \put (-242,60) {$\frac{16n-2}{8n-3}$}
        \put (-170,60) {$-\frac{16n-2}{8n-3}$}      
        \put (-10,62) {$-(8n-2)$}
        \put (-55,17) {$\frac{32n^2-12n+1}{16n^2-6n+1}$}
        \put (-70,130) {$0$}
       \caption{Equivalent surgery diagrams for 0-surgery on $J_n$ in $M_n$.}
\label{fig:surgeryJn}
\end{figure}
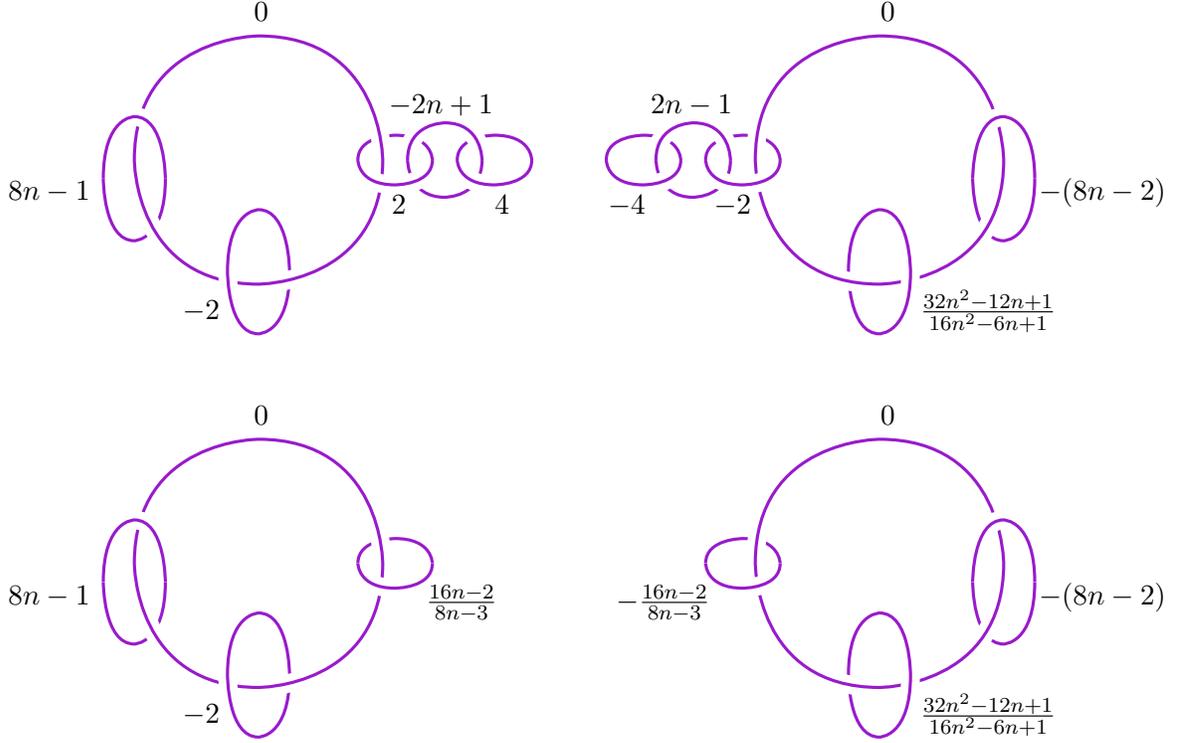

\begin{proof}[Proof of Theorem~\ref{thm:main2}]
Let $L_n$ be the link defined in Proposition~\ref{prop:surg-desc}.  Then the result of 0-surgery on $L_n$ is $Y_n \# Z_n$.  If $L_n$ is weakly handleslide equivalent to a split link, then each of $Y_n$ and $Z_n$ can be obtained by surgery on a knot in $S^3$, contradicting Theorem~\ref{thm:deus}.
\end{proof}

\section{Another perspective on the links $L_n$}\label{sec:proof2}

In this section, we offer insights about the connections between the links $L_n$ and the constructions in~\cite{GST} and~\cite{MeZ} for further topological context.  Note that in our construction, the knot $K_n$, as the connected sum of two fibered knots, is itself fibered.  Scharlemann and Thompson placed strong restrictions on 2-component links $L = K \cup J$ with reducible surgeries in the case that one component is fibered.  The following proposition is a direct consequence of Corollary 4.2 of~\cite{ST} (see also Theorem 3.3 of~\cite{GST} and Lemma 3.10 of~\cite{MeZ}):

\begin{proposition}\label{prop:fiberlay}
Suppose $L = K \cup J$ is a 2-component framed link such that $K$ is fibered and surgery on $L$ yields $Y \# Z$, a connected sum of homology $S^1 \times S^2$'s.  Then either $L$ is handleslide equivalent to a split link, or there is a sequence of handle-slides of $J$ over $K$ which takes $L$ to a link $L' = K \cup J'$ such that $J'$ lies in a fiber of $K$ with the zero-framing induced by the fiber.
\end{proposition}

\begin{proof}
Such a link satisfies the hypotheses of Corollary 4.2 of~\cite{ST}, and of the four resulting conclusions, only two are possible based on the algebraic topology asserted in Lemma~\ref{lem:alg-unlink}.  Let $M$ denote the result of surgery on the fibered component $K$.  Then either $J$ (viewed as a knot in $M$) is isotopic into a 3-ball, in which case $L$ is handleslide equivalent to a split link, or $J$ is isotopic into a (closed) fiber of $M$, with surface framing equal to the framing of $J$.  It follows that as a link in $S^3$, the knot $J$ can be pushed onto the fiber of $K$ modulo possible handleslides over $K$.
\end{proof}

Now, consider the (fibered) torus knot $T_{p,q}$.  An explicit description of the fibration of $S^3 \setminus \eta(T_{p,q})$ is given in~\cite{MeZ}:  The fiber $F_{p,q}$ is a genus $(p-1)(q-1)/2$ surface with one boundary component.  Let $A_{p,q}$ be a $2pq$-gon with an open disk removed, so that $A_{p,q}$ is a topological annulus.  With appropriate edge identifications, the fiber $F_{p,q}$ is obtained naturally as a quotient of $A_{p,q}$, and the monodromy $\varphi_{p,q}$ is induced by a clockwise $2\pi/pq$-rotation of $A_{p,q}$.  Typically, the monodromy of a knot complement is chosen to fix the boundary of the fiber.  However, in this case $S^3 \setminus \eta(T_{p,q})$ is Seifert fibered over the disk $D^2_{p,q}$, and if $\rho_{p,q}$ is the natural projection map $\rho_{p,q} : F_{p,q} \rightarrow D^2_{p,q}$, then our choice of $\varphi_{p,q}$ (which rotates the boundary of $F_{p,q}$) is equivariant in the sense that $\rho_{p,q} = \rho_{p,q} \circ \varphi_{p,q}$.

Recall that $K_n = -T_{8n-1,8n-2} \# T_{32n^2-12n+1,2}$.  From the proof of Proposition~\ref{prop:surg-desc}, we have that the manifold $M$ obtained by 0-surgery on $K_n$ is a Seifert fibered space over $S^2$ with four singular fibers.  Notably, $(8n-1)(8n-2) = 2(32n^2-12n+1)$, so both the monodromies $\varphi_{8n-1,8n-2}$ and $-\varphi_{32n^2-12n+1,2}$ induce rotations of the same magnitude on their respective fiber surfaces.  It follows that the fiber $\widehat F$ for $M_n$ is obtained by gluing the two knot fibers, the two monodromies can be pasted together to get the (closed) monodromy $\widehat \varphi$ for $M_n$, and the base space $S$ for the Seifert fibration of $M_n$ is obtained by gluing the base spaces for the two knot summands.  By construction, if $\rho: M_n \rightarrow S$ is the quotient to the base orbifold, the monodromy $\widehat \varphi$ satisfies $\rho = \rho \circ \widehat \varphi$.

Since the surface $\wh F$ is horizontal with respect to the Seifert fibration of $M_n$, the restriction $\rho|_{\wh F}:\wh F \rightarrow S$ is a branched covering map, whose branch points coincide with the orbifold points (with multiplicities) of $S$.  We invoke a useful lemma about branched coverings, whose proof is left to the reader.  The interested reader is encouraged to consult~\cite{scott}, for instance, regarding branched coverings and Seifert fibered spaces.

\begin{lemma}\label{lem:branch}
Suppose $F_+$ is a connected, compact surface, $S_+$ is a disk, and $p:F_+ \rightarrow S_+$ is a branched covering with two branch points of orders $a$ and $b$.  Then $F_+$ has $\gcd(a,b)$ boundary components.
\end{lemma}

With this setup, we give a second proof of Proposition~\ref{prop:surg-desc}.

\begin{proof}[Alternate proof of Proposition~\ref{prop:surg-desc}]
Choose a curve $\lambda$ in $S$ that separates the cone points of orders $2$ and $8n-1$ from the cone points of orders $8n-2$ and $32n^2-12n+1$, as shown in Figure~\ref{fig:pillow}, and let $\Lambda = \rho^{-1}(\lambda)$, so that $\Lambda$ is an essential vertical torus in $M_n$, the manifold obtained by 0-surgery on $K_n$.  Let $J_n$ be any curve in the intersection $\Lambda \cap \widehat F = \rho_{\wh F}^{-1}(\lambda)$, and endow $J_n$ with the framing determined by $\Lambda$ (which coincides with the framing determined by $\wh F$).  Performing surgery on $J_{n}$ with this framing corresponds to cutting along $\Lambda$, attaching two $3$-dimensional $2$-handles along the respective copies of $J_{n}$ in the boundary, and re-identifying the resulting boundary components.  (See, for example, the ``Surgery Principle" of \cite{GST}.)  In this way, we can see that the result $N_n$ of surgery on $J_n$ in $M_n$ can be expressed as a connected sum $Y_n \# Z_n$, and we seek to identify these summands.  %$N = Y \# Z$, where $Y$ and $Z$ are Dehn fillings of since the result of attaching a 2-handle to $\Lambda$ is a 2-sphere, namely the decomposing sphere in the connected sum.  We note, however, that $N$ may also be obtained by cutting along $\Lambda$, attaching a solid torus to each boundary component, removing a $3$-ball from each solid torus, and identifying the resulting spherical boundary components.

\begin{figure}[h!]
    \centering
    \phantom{hi} \includegraphics[width=0.3\textwidth]{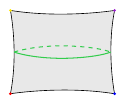}
        \put (-168,100) {$8n-1$}
        \put (-168,10) {$8n-2$}
        \put (-3,100) {$2$}
        \put (-3,10) {$32n^2-12n+1$}
\caption{The curve $\lambda$ in $S$ separates cone points as specified.}
\label{fig:pillow}
\end{figure}
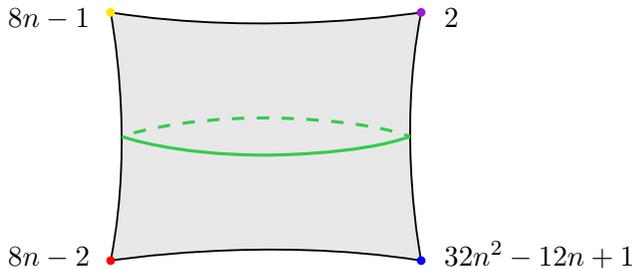

Cutting $M_n$ open along $\Lambda$ respects the Seifert fibration, yielding two Seifert fibered spaces, both of which fiber over $D^2$ and which have exceptional fibers of orders $2$ and $8n-1$ or $8n-2$ and $32n^2-12n+1$, respectively.  Recall that $M_n = (S^2; -\frac{2}{1}, \frac{8n-1}{1}, -\frac{8n-2}{1}, \frac{32n^2-12n+1}{16n^2-6n+1})$.  Therefore, we see that $Y_n = (S^2; -\frac{2}{1},\frac{8n-1}{1},\frac{p}{q})$ and $Z_n = (S^2; \frac{r}{s},-\frac{8n-2}{1},\frac{32n^2-12n+1}{16n^2-6n+1})$.  Let $S_+$ be the closure of $S \setminus \lambda$ containing the orbifold points of orders $8n-1$ and 2, and let $F_+$ be any component of $\rho_{\wh F}^{-1}(S_+)$.  By Lemma~\ref{lem:branch}, we have that $F_+$ has $\gcd(8n-1,2) = 1$ boundary component.  We conclude that $J_n$ is separating in $\wh F$.  It follows that $J_n$ is nullhomologous in $M_n$ via a subsurface of $\widehat F$, and the surface framing of $J_n$ agrees with the 0-framing.

%\textcolor{red}{Consider all curves of intersection of $\Lambda$ with $\widehat F$.  Noting that $8n-1$ and $2$ are relatively prime, while the gcd of $8n-2$ and $32n^2 - 12n+1$ is $4n-1$, it follows that $\Lambda \cap \widehat F$ cuts $\widehat F$ into $4n$ pieces, $4n-1$ of which have one boundary component (on the side of $Y$), and one of which has $4n-1$ boundary components (on the side of $Z$).  This implies that $J_n$ separate $\widehat F$, so that $J_n$ is nullhomologous in $M$ via a subsurface of $\widehat F$, and the surface framing of $J_n$ agrees with the 0-framing.}

As $K_n$ is also 0-framed and $J_n$ is contained in a fiber, the linking number of $K_n$ and $J_n$ is zero.  We conclude that both $Y_n$ and $Z_n$ are homology $S^1 \times S^2$'s, and in particular, both have Euler number zero.  By~\cite{Mart}, we can compute $\frac{p}{q} = \frac{16n-2}{8n-3}$ and $\frac{r}{s} = -\frac{16n-2}{8n-3}$, completing the proof.
\end{proof}

In addition to connecting with the references~\cite{GST} and~\cite{MeZ} and providing the perspective with which we first conceived of our construction, the alternative proof also yields a concrete description of the links $L_n$.  In the case $n=1$, we have $K_1 = -T_{7,6} \# T_{21,2}$.  The closed fiber $\widehat F$ is genus-25 surface which can be constructed as the quotient of an annulus with two 84-gon boundary components, the monodromy $\varphi$ is a $2\pi/42$ rotation, and a curve $\lambda$ in the base space $S$ separating the cone points in the prescribed way lifts to three separating curves cutting $\widehat F$ into four components.  One of these lifts is shown in Figure~\ref{fig:crazy} below.  Following the procedures outlined in~\cite{MeZ} (building on work of Scharlemann in~\cite{Scharl}), one could conceivably draw $J_1$ on a fiber of $K_1$ in $S^3$, but we expect a realization of this link (even in the simplest case of $n=1$) to have thousands of crossings, so we omit an explicit picture here.

\begin{figure}[h!]
    \centering
    \phantom{hi} \includegraphics[width=1\textwidth]{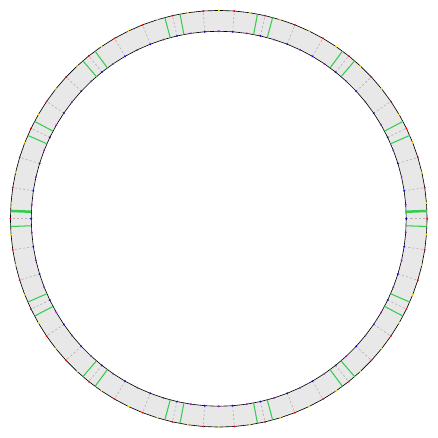}
\caption{A lift of $\lambda$ to $\widehat F$.  Here $\widehat F$ is the quotient of the annulus via edge identifications.  Around both boundary curves, starting with the thickened green arc on the right, the edges are labeled $(i,j)^+$, $(i,j+1)^-$, $(i+1,j+1)^+$, continuing this pattern proceeding counterclockwise.  On the outer boundary, the index $i$ is taken mod 7 and the index $j$ is taken mod 6.  On the inner boundary, the index $i$ is taken mod 2 and the index $j$ is taken mod 21.  For instance, the two thickened green segments are glued along their inner boundaries.  This figure generalizes the construction shown in~\cite[Figure 6]{MeZ}.  Note that the 42 fundamental domains of the covering from $\widehat F$ to $S$ are delineated by dotted lines, and the vertices are color-coded to indicate the preimages of the cone points in Figure~\ref{fig:pillow}.}
\label{fig:crazy}
\end{figure}

\bibliographystyle{alpha}
\bibliography{references.bib}

\end{document}